\documentclass[11pt]{amsart}

\usepackage{amsmath, amssymb, amsthm}
\usepackage{mathtools}
\usepackage{tikz-cd}
\usepackage{hyperref}
\usepackage{cleveref}
\usepackage{enumitem}
\usepackage[margin=1in]{geometry}
\usepackage{booktabs}

\newtheorem{maintheorem}{Theorem}

\theoremstyle{plain}
\newtheorem{theorem}{Theorem}[section]
\newtheorem{proposition}[theorem]{Proposition}
\newtheorem{lemma}[theorem]{Lemma}
\newtheorem{corollary}[theorem]{Corollary}

\theoremstyle{definition}
\newtheorem{definition}[theorem]{Definition}
\newtheorem{example}[theorem]{Example}
\newtheorem{notation}[theorem]{Notation}

\newtheorem{remark}[theorem]{Remark}

\DeclareMathOperator{\Spec}{Spec}
\newcommand{\Hom}{\mathrm{Hom}}
\newcommand{\colim}{\mathrm{colim}}
\newcommand{\Set}{\mathbf{Set}}

\newcommand{\CRing}{\mathbf{CRing}}
\newcommand{\CMonoid}{\mathbf{CMonoid}}
\newcommand{\GSet}{\Gamma\mathbf{Set}_*}

\newcommand{\FAlg}{\mathbb{F}_1\mathbf{Alg}}
\newcommand{\AffSch}{\mathbf{AffSch}}
\newcommand{\FF}{\mathbb{F}_1}
\newcommand{\ZZ}{\mathbb{Z}}
\newcommand{\NN}{\mathbb{N}}

\hypersetup{
    colorlinks=true,
    linkcolor=blue,
    citecolor=blue,
    urlcolor=blue
}

\begin{document}

\title{Localization and Affine Schemes over $\mathbb{F}_1$}

\author{Luqiao Xu}
\address{Johns Hopkins University}
\email{lxu46@jhu.edu}

\date{\today}

\begin{abstract}
We construct an affine geometry for commutative
$\mathbb F_1$-algebras in the framework of
Connes and Consani.
Using their localization construction, we equip the
Deitmar spectrum of the underlying multiplicative monoid
$A(1_+)$ with a sheaf $\mathcal O_A$ of
$\mathbb F_1$-algebras. We prove that
$\mathcal O_A(D(f))\cong A_f$ and, in particular, that
global sections recover $A$. The resulting spectrum
construction yields an anti-equivalence between commutative
$\mathbb F_1$-algebras and absolute affine schemes.
We compare our choice of covering families with the
Connes--Consani topology and explain its role in affine
reconstruction. Finally, we describe base change to
$\mathbb Z$, under which the absolute affine scheme
associated to the Eilenberg--MacLane algebra $HR$
recovers the classical affine scheme $\operatorname{Spec}R$.
\end{abstract}

\maketitle

\tableofcontents

\section{Introduction}

Classical affine algebraic geometry associates to a commutative
ring $R$ a topological space $\operatorname{Spec}R$ equipped
with a structure sheaf obtained by localization. Its global
sections recover $R$, and the spectrum construction establishes
an anti-equivalence between commutative rings and affine schemes.
The purpose of this paper is to establish an analogous
construction for commutative $\mathbb F_1$-algebras in the
framework of Connes and Consani.

In this framework, $\mathbb F_1$-algebras are monoid objects
in the symmetric monoidal category of $\Gamma$-sets
\cite{connes2016absolute, CC21}. Their first level $A(1_+)$
is a pointed multiplicative monoid, while their higher levels
encode additive information. This setting accommodates both
pointed commutative monoids, through the spherical functor
$M\mapsto\mathbb S M$, and commutative rings, through the
Eilenberg--MacLane functor $R\mapsto HR$.

Connes and Consani construct a spectrum for such an algebra
as a Grothendieck site, whose covering families depend on
partitions encoded by its higher levels \cite{CC21}.
They also define localization with respect to multiplicative
subsets of $A(1_+)$, providing a structure presheaf on this site.
In general, this presheaf requires sheafification, and the
global sections of the resulting structure sheaf need not
recover the original algebra.

We use the same localization construction but choose a
different topology. For a commutative $\mathbb F_1$-algebra
$A$, we set
\[
X_A:=\operatorname{Spec}_{\mathrm{Deitmar}}A(1_+)
\]
and equip $X_A$ with a sheaf $\mathcal O_A$ of commutative
$\mathbb F_1$-algebras whose sections are locally represented
by fractions. Thus the topology depends only on the
multiplicative monoid $A(1_+)$, while the structure sheaf
retains the higher levels of $A$.

This construction is motivated in part by the spectrum
construction in spectral algebraic geometry, where the
topology is determined by $\pi_0A$ and the structure sheaf
retains the richer algebraic object $A$ \cite{SAG}.
Here, evaluation at $1_+$ supplies the underlying
multiplicative monoid in place of the underlying ring
$\pi_0A$.

An essential choice is that prime ideals of $A(1_+)$ are
not required to be closed under the additive operations
encoded by $A$. For $A=HR$, imposing additive closure
selects the ordinary prime ideals of $R$, whereas our
multiplicative spectrum generally has additional points.
For more general $\mathbb F_1$-algebras, additive closure
can exclude points needed for global recovery, and may
even leave no prime ideals for a nonzero algebra.
Our choice is guided by the aim of recovering every
commutative $\mathbb F_1$-algebra as the global sections
of its spectrum. We discuss this choice in
Remark~\ref{rem:prime} and compare the two constructions
in \S\ref{subsec:comparison-cc}.

We first recall the localization of Connes and Consani
\cite[\S7.1]{CC21} and establish the universal property
used throughout the paper.

\begin{maintheorem}[Localization]
\label{thm:main-localization}
Let $A$ be a commutative $\mathbb F_1$-algebra and let
$S\subseteq A(1_+)$ be a multiplicatively closed subset.
There is a commutative $\mathbb F_1$-algebra $S^{-1}A$
and a morphism $A\to S^{-1}A$ universal among morphisms
sending every element of $S$ to a unit at the first level.

For a commutative ring $R$ and a pointed commutative
monoid $M$, localization is compatible with the
Eilenberg--MacLane and spherical constructions:
\[
S^{-1}(HR)\cong H(S^{-1}R),
\qquad
S^{-1}(\mathbb S M)\cong\mathbb S(S^{-1}M),
\]
for multiplicatively closed subsets of the respective
underlying monoids.
\end{maintheorem}

Using this localization, we define
\[
\operatorname{Spec}A:=(X_A,\mathcal O_A).
\]
The principal result is that its structure sheaf has
the expected stalks and sections, without any loss of
information through sheafification.

\begin{maintheorem}[Structure sheaf]
\label{thm:main-structure-sheaf}
Let $A$ be a commutative $\mathbb F_1$-algebra.
There are canonical isomorphisms of
$\mathbb F_1$-algebras
\[
\mathcal O_{A,\mathfrak p}\cong A_{\mathfrak p}
\qquad
(\mathfrak p\in X_A)
\]
and
\[
\mathcal O_A(D(f))\cong A_f
\qquad
(f\in A(1_+)).
\]
In particular,
\[
\Gamma(X_A,\mathcal O_A)\cong A.
\]
\end{maintheorem}

To formulate the geometric reconstruction, we consider
spaces equipped with sheaves of commutative
$\mathbb F_1$-algebras and morphisms whose induced stalk
maps are local at the first level, meaning that they
reflect multiplicative units. An \emph{absolute affine
scheme} is such a space isomorphic to
$\operatorname{Spec}A$ for some commutative
$\mathbb F_1$-algebra $A$.

\begin{maintheorem}[Affine reconstruction]
\label{thm:main-antiequiv}
The spectrum construction defines an equivalence
\[
\operatorname{Spec}:
\FAlg^{\mathrm{op}}
\xrightarrow{\sim}
\AffSch_{\mathbb F_1},
\]
with quasi-inverse given by global sections.
Equivalently, for commutative $\mathbb F_1$-algebras
$A$ and $B$, there is a natural bijection
\[
\operatorname{Hom}_{\FAlg}(A,B)
\cong
\operatorname{Hom}_{\AffSch_{\mathbb F_1}}
(\operatorname{Spec}B,\operatorname{Spec}A).
\]
\end{maintheorem}

For spherical algebras $\mathbb S M$, the construction
recovers affine monoid geometry through the spherical
functor. The Eilenberg--MacLane functor likewise induces
a fully faithful embedding of classical affine schemes
into absolute affine schemes. This embedding generally
changes the underlying topological space. Nevertheless,
the classical affine scheme is recovered by base change:
\[
(\operatorname{Spec}HR)\otimes_{\mathbb F_1}\mathbb Z
\cong
\operatorname{Spec}R.
\]
More generally, the algebraic adjunction
$(-\otimes_{\mathbb F_1}\mathbb Z)\dashv H$
induces a base change functor on absolute affine schemes.
Thus every classical affine scheme admits a functorial
absolute affine model in this framework.

\subsection*{Outline of the paper}

\Cref{sec:background} reviews $\Gamma$-sets,
$\mathbb F_1$-algebras, and affine monoid schemes.
\Cref{sec:localization} recalls localization and proves
its universal property.
\Cref{sec:sheaves} develops the sheaf-theoretic constructions
needed for sheaves of $\mathbb F_1$-algebras.
\Cref{sec:spectrum} constructs the spectrum and its structure
sheaf, proves the recovery of sections, and compares the
construction with that of Connes and Consani.
\Cref{sec:antiequiv} proves functoriality and affine
reconstruction.
Finally, \Cref{sec:applications} describes base change to
classical affine schemes and its basic examples.

\subsection*{Acknowledgments}

This paper is based on part of the author's Ph.D. thesis
\cite{xu-thesis}. The author thanks his doctoral advisor,
Caterina Consani, for her guidance.

\section{Background}\label{sec:background}

The concept of a geometry over the hypothetical ``field with one element'' can be traced back to the work of Jacques Tits, through the study of spherical buildings \cite{Tits}. Since then, a number of approaches have been proposed to make this notion precise, notably Deitmar's monoid schemes \cite{deitmar2006schemesf1}, Lorscheid's blueprints \cite{blueprint}, Connes--Consani's $\Gamma$-sets \cite{connes2016absolute, CC21}.

In this section, we recall the algebraic framework of Connes and Consani
and the basic facts about affine monoid schemes used
in our construction. Throughout the paper, rings are
commutative with identity, and pointed commutative
monoids have an absorbing element $0$ and an identity
element $1$. Morphisms preserve both elements.
We allow the zero ring and the trivial pointed monoid,
in which $0=1$.

\subsection{$\Gamma$-sets and $\mathbb F_1$-algebras}
\label{subsec:gamma-recap}

We follow the conventions of
\cite{connes2016absolute, CC21, dundas2012local}.

\begin{definition}
The category $\Gamma^{\mathrm{op}}$ has objects
\[
n_+=\{0,1,\ldots,n\},
\qquad n\geq 0,
\]
with basepoint $0$, and basepoint-preserving maps as
morphisms. A \emph{$\Gamma$-set} is a functor
\[
F:\Gamma^{\mathrm{op}}\longrightarrow\Set_*
\]
such that $F(0_+)=\{*\}$. We call $F(n_+)$ the
\emph{$n$th level} of $F$, and write $\GSet$ for the
category of $\Gamma$-sets and natural transformations.
\end{definition}

The smash product of finite pointed sets induces a
symmetric monoidal structure on $\GSet$ by Day
convolution. Its unit is the functor
\[
\FF(n_+)=n_+,
\]
acting as the identity on morphisms.
This monoidal structure is closed, with internal Hom
given by
\[
\underline{\Hom}_{\GSet}(F,G)(n_+)
=
\Hom_{\GSet}\bigl(F,G(n_+\wedge-)\bigr).
\]

Two constructions relate $\Gamma$-sets to familiar
algebraic objects.

\begin{definition}[Eilenberg--MacLane functor]
For a commutative monoid $M$, written additively with
identity $0$, define
\[
HM(n_+)=M^n,
\]
pointed by the zero tuple. For a pointed map
$f:n_+\to k_+$, set
\[
HM(f)(m_1,\ldots,m_n)
=
\left(
\sum_{f(i)=1}m_i,\ldots,
\sum_{f(i)=k}m_i
\right),
\]
where an empty sum is $0$.
This defines a fully faithful functor
\[
H:\CMonoid\longrightarrow\GSet.
\]
\end{definition}

\begin{definition}[Spherical functor]
For a pointed set $X$, define
\[
\mathbb S X(n_+)=X\wedge n_+,
\]
with structure maps induced by those of $n_+$.
Thus $\mathbb S X(n_+)$ is the wedge of $n$ copies
of $X$. This construction defines a fully faithful
functor
\[
\mathbb S:\Set_*\longrightarrow\GSet,
\]
left adjoint to evaluation at $1_+$:
\[
\Hom_{\GSet}(\mathbb S X,F)
\cong
\Hom_{\Set_*}(X,F(1_+)).
\]
\end{definition}

These constructions illustrate the role of the higher
levels. In $HM$, an element of $HM(n_+)$ records an
$n$-tuple, and the map $n_+\to1_+$ sending every
nonzero element to $1$ records its sum.
In $\mathbb S X$, an element at level $n$ has at most
one non-basepoint coordinate, so only sums involving
a single nonzero term are represented.

\begin{definition}[$\mathbb F_1$-algebras]
\label{def:f1-algebra-recap}
An \emph{$\mathbb F_1$-algebra} is a monoid object in
$(\GSet,\wedge,\FF)$. It consists of a $\Gamma$-set
$A$, a multiplication
\[
\mu:A\wedge A\longrightarrow A,
\]
and a unit
\[
\eta:\FF\longrightarrow A,
\]
satisfying the associativity and unit axioms.
It is \emph{commutative} if its multiplication is
compatible with the symmetry of the smash product.

We denote by $\FAlg$ the category of commutative
$\mathbb F_1$-algebras.
\end{definition}

By the defining property of Day convolution,
multiplication can be expressed as a natural family
of maps
\[
A(m_+)\wedge A(n_+)
\longrightarrow A(m_+\wedge n_+).
\]
In particular, $A(1_+)$ is a pointed commutative
monoid. Its identity is
$1_A=\eta_{1_+}(1)$, and its absorbing element is
the basepoint.

The first level also acts on every higher level:
\[
A(1_+)\wedge A(n_+)\longrightarrow A(n_+).
\]
Naturality gives
\[
A(f)(sa)=sA(f)(a)
\]
for $s\in A(1_+)$, $a\in A(n_+)$, and
$f:n_+\to k_+$. This compatibility is the basis
of the localization construction in
\S\ref{sec:localization}.

\begin{proposition}[\cite{connes2016absolute}]
\label{prop:functors-recap}
\begin{enumerate}[label=(\roman*)]
\item
For a commutative ring $R$, the $\Gamma$-set $HR$
is a commutative $\mathbb F_1$-algebra with
multiplication
\[
(r_i)_{i=1}^m\wedge(s_j)_{j=1}^n
\longmapsto
(r_is_j)_{(i,j)}.
\]
This construction defines a fully faithful functor
\[
H:\CRing\longrightarrow\FAlg.
\]

\item
For a pointed commutative monoid $M$, multiplication
in $M$ induces a commutative $\mathbb F_1$-algebra
structure on $\mathbb S M$.
This construction defines a fully faithful functor
\[
\mathbb S:\CMonoid_*\longrightarrow\FAlg.
\]
\end{enumerate}
\end{proposition}

At the algebraic level, these functors belong to
the adjunctions
\[
\mathbb S\dashv(-)(1_+),
\qquad
(-\otimes_{\FF}\ZZ)\dashv H.
\]
Here $-\otimes_{\FF}\ZZ$ takes values in commutative
rings; its construction is recalled in
\S\ref{sec:applications}. For further details,
see \cite{xu-hyper, beardsley2025}.

Their composite sends a pointed commutative monoid
$M$ to its contracted monoid ring:
\[
(\mathbb S M)\otimes_{\FF}\ZZ
\cong
\ZZ_0[M]:=\ZZ[M]/([0]).
\]
In this expression, $\ZZ[M]$ is the monoid ring,
and $[0]$ denotes the basis element corresponding
to the absorbing element of $M$.
The functor $M\mapsto\ZZ_0[M]$ is left adjoint
to the functor taking a ring to its underlying
pointed multiplicative monoid.

\subsection{Deitmar's monoid schemes}
\label{subsec:deitmar}

We recall affine monoid geometry in the pointed
setting, following
\cite{deitmar2006schemesf1, everyone}.
It supplies the underlying topological spaces
for our spectra.

\subsubsection{Localization}

\begin{definition}\label{def:monoid-localization}
Let $M$ be a pointed commutative monoid and
$S\subseteq M$ a multiplicatively closed subset,
so that $1\in S$ and $SS\subseteq S$.
The \emph{localization} of $M$ at $S$ is
\[
S^{-1}M
=
\left\{
\frac{a}{s}\;\middle|\;a\in M,\ s\in S
\right\}\bigg/\sim,
\]
where
\[
\frac{a}{s}\sim\frac{b}{t}
\quad\Longleftrightarrow\quad
\text{there exists }u\in S
\text{ such that }uta=usb.
\]
Multiplication is given by
\[
\frac{a}{s}\frac{b}{t}=\frac{ab}{st},
\]
with identity $1/1$ and absorbing element $0/1$.
\end{definition}

\begin{proposition}[\cite{deitmar2006schemesf1}]
\label{prop:monoid-localization-univ}
The canonical morphism
\[
\iota:M\longrightarrow S^{-1}M,
\qquad a\longmapsto a/1,
\]
sends every element of $S$ to a unit.
For any morphism $\varphi:M\to N$ of pointed
commutative monoids satisfying
$\varphi(S)\subseteq N^\times$, there is a unique
morphism
\[
\overline{\varphi}:S^{-1}M\longrightarrow N
\]
such that $\overline{\varphi}\circ\iota=\varphi$.
\end{proposition}

For $f\in M$, we write
\[
M_f=\{1,f,f^2,\ldots\}^{-1}M.
\]
If $0\in S$, then $S^{-1}M$ is the trivial
pointed monoid.

\subsubsection{Prime ideals and the spectrum}

\begin{definition}\label{def:prime-ideal-monoid}
An \emph{ideal} of a pointed commutative monoid $M$
is a nonempty subset $I\subseteq M$ satisfying
$MI\subseteq I$. In particular, every ideal
contains $0$.

An ideal $\mathfrak p$ is \emph{prime} if it is
proper and
\[
ab\in\mathfrak p
\quad\Longrightarrow\quad
a\in\mathfrak p\ \text{or}\ b\in\mathfrak p.
\]
Equivalently, its complement is multiplicatively
closed.

The \emph{prime spectrum}
$\Spec_{\mathrm{Deitmar}}M$ is the set of prime
ideals of $M$, with the topology whose basic
open sets are
\[
D(f)=\{\mathfrak p\mid f\notin\mathfrak p\},
\qquad f\in M.
\]
Its closed sets are
\[
V(I)=\{\mathfrak p\mid I\subseteq\mathfrak p\},
\]
where $I$ ranges over the ideals of $M$.
\end{definition}

For a prime ideal $\mathfrak p$, we write
\[
M_{\mathfrak p}=(M\setminus\mathfrak p)^{-1}M.
\]
The spectrum of the trivial pointed monoid is
empty.

\begin{proposition}\label{prop:monoid-spectrum-properties}
Let $M$ be a nontrivial pointed commutative monoid.
Then:
\begin{enumerate}[label=(\roman*)]
\item
The set
\[
\mathfrak m=M\setminus M^\times
\]
is the unique maximal ideal of $M$, and it is prime.

\item
For ideals $I,J$ and a nonempty family of ideals
$\{I_\alpha\}$, one has
\[
V(I)\cup V(J)=V(I\cap J),
\qquad
\bigcap_\alpha V(I_\alpha)
=
V\left(\bigcup_\alpha I_\alpha\right).
\]

\item
For $f,g\in M$,
\[
D(f)\cap D(g)=D(fg).
\]

\item
The sets $D(f)$ form a basis for the topology.
\end{enumerate}
\end{proposition}

Unlike ideals of rings, ideals of monoids require
no additive closure. Consequently, a nonempty
union of monoid ideals is again an ideal.
The existence of the largest prime
$M\setminus M^\times$ also has a useful topological
consequence: every open neighborhood of this point
is the whole spectrum. Indeed, a basic open $D(f)$
containing it has $f\in M^\times$, and hence
$D(f)=\Spec_{\mathrm{Deitmar}}M$.

\begin{definition}\label{def:radical-monoid}
For an ideal $I\subseteq M$, its \emph{radical} is
\[
\sqrt I
=
\{a\in M\mid a^n\in I
\text{ for some }n\geq1\}.
\]
\end{definition}

The radical of an ideal is the intersection of
the prime ideals containing it
\cite{deitmar2006schemesf1}. We will use the
following consequences.

\begin{proposition}\label{thm:monoid-nullstellensatz}
Let $M$ be a pointed commutative monoid.
\begin{enumerate}[label=(\roman*)]
\item
If $V(I)\subseteq V(J)$, then
$\sqrt J\subseteq\sqrt I$.

\item
If $V(I)\subseteq V(\langle f\rangle)$, then
$f\in\sqrt I$.
\end{enumerate}
Here $\langle f\rangle=fM$ is the principal ideal
generated by $f$.
\end{proposition}

In particular,
\[
D(f)\subseteq D(g)
\quad\Longleftrightarrow\quad
f^n\in gM
\text{ for some }n\geq1,
\]
and $D(f)=\varnothing$ if and only if $f$ is
nilpotent.

\subsubsection{The structure sheaf}

\begin{definition}\label{def:deitmar-structure-sheaf}
Let $X=\Spec_{\mathrm{Deitmar}}M$.
For an open set $U\subseteq X$, define
$\mathcal O_M(U)$ to be the pointed monoid of
functions
\[
s:U\longrightarrow
\coprod_{\mathfrak p\in U}M_{\mathfrak p}
\]
such that $s(\mathfrak p)\in M_{\mathfrak p}$
and $s$ is locally a quotient.

Explicitly, for every $\mathfrak p\in U$,
there must exist an open neighborhood
$V\subseteq U$ of $\mathfrak p$ and elements
$a,f\in M$ such that $f\notin\mathfrak q$
for all $\mathfrak q\in V$ and
\[
s(\mathfrak q)=\frac{a}{f}
\quad\text{in }M_{\mathfrak q}
\qquad(\mathfrak q\in V).
\]
Multiplication and restriction are defined
pointwise. These assignments form the
\emph{structure sheaf} $\mathcal O_M$.
\end{definition}

\begin{theorem}[\cite{deitmar2006schemesf1}]
\label{thm:deitmar-sheaf-properties}
Let $M$ be a pointed commutative monoid.
\begin{enumerate}[label=(\roman*)]
\item
For every prime ideal $\mathfrak p$,
\[
\mathcal O_{M,\mathfrak p}\cong M_{\mathfrak p}.
\]

\item
For every $f\in M$,
\[
\mathcal O_M(D(f))\cong M_f.
\]

\item
Global sections recover $M$:
\[
\Gamma(\Spec_{\mathrm{Deitmar}}M,\mathcal O_M)
\cong M.
\]
\end{enumerate}
\end{theorem}

\subsubsection{Monoidal spaces and affine monoid schemes}

\begin{definition}\label{def:monoidal-space}
A \emph{monoidal space} is a pair
$(X,\mathcal O_X)$ consisting of a topological
space and a sheaf of pointed commutative monoids.

A morphism
\[
(f,f^\#):(X,\mathcal O_X)\longrightarrow
(Y,\mathcal O_Y)
\]
consists of a continuous map $f:X\to Y$ and a
morphism of sheaves
\[
f^\#:\mathcal O_Y\longrightarrow f_*\mathcal O_X.
\]
It is \emph{local} if, for every $x\in X$, the
induced stalk map satisfies
\[
(f^\#_x)^{-1}(\mathcal O_{X,x}^{\times})
=
\mathcal O_{Y,f(x)}^{\times}.
\]
\end{definition}

\begin{definition}\label{def:affine-monoid-scheme}
An \emph{affine monoid scheme} is a monoidal
space isomorphic to
\[
(\Spec_{\mathrm{Deitmar}}M,\mathcal O_M)
\]
for some pointed commutative monoid $M$.
Morphisms of affine monoid schemes are local
morphisms of monoidal spaces.
We denote their category by $\mathbf{MonSch}$.
\end{definition}

\begin{theorem}[\cite{deitmar2006schemesf1}]
\label{thm:deitmar-antiequiv}
The spectrum construction defines an equivalence
\[
\Spec_{\mathrm{Deitmar}}:
\CMonoid_*^{\mathrm{op}}
\xrightarrow{\sim}
\mathbf{MonSch},
\]
with quasi-inverse given by global sections.
\end{theorem}

Our construction retains the underlying monoid
spectrum and replaces its structure sheaf by a
sheaf carrying the full $\mathbb F_1$-algebra
structure. We begin by recalling the localization
needed to define this sheaf.

\section{Commutative algebra over $\mathbb{F}_1$: Localization}\label{sec:localization}

In this section, we recall the localization construction
of Connes and Consani \cite[\S7.1]{CC21} and establish
the universal property used in our spectrum construction.
The key observation is that elements of $A(1_+)$ act
naturally on every level of $A$:
\[
A(1_+)\wedge A(n_+)\longrightarrow A(n_+).
\]
This allows us to localize all levels simultaneously
while preserving the $\Gamma$-set and algebra structures.

\subsection{Definition and basic properties}

\begin{definition}\label{def:multiplicatively-closed}
Let $A$ be a commutative $\mathbb{F}_1$-algebra. A subset $S \subseteq A(1_+)$ is \emph{multiplicatively closed} if:
\begin{enumerate}[label=(\roman*)]
\item $1_A \in S$ (contains the unit)
\item If $s, t \in S$, then $s \cdot t \in S$ (closed under multiplication)
\end{enumerate}
\end{definition}

\begin{remark}
We require $S \subseteq A(1_+)$ rather than allowing elements from arbitrary levels. This is natural: for an element to be ``invertible,'' it must live at level $1$ so that multiplication by its inverse could yield the unit $1_A \in A(1_+)$.
\end{remark}

We now define localization:

\begin{definition}\label{def:localization}
Let $A$ be a commutative $\mathbb{F}_1$-algebra and $S \subseteq A(1_+)$ a multiplicatively closed subset. We define the \emph{localization} $S^{-1}A$ as follows.

For each level $n_+ \in \Gamma^{\mathrm{op}}$, we set:
$$S^{-1}A(n_+) := \left\{\frac{a}{s} \,\bigg|\, a \in A(n_+), s \in S\right\} \bigg/ \sim$$
where the equivalence relation is:
$$\frac{a_1}{s_1} \sim \frac{a_2}{s_2} \quad\iff\quad \exists t \in S: ta_1s_2 = ta_2s_1 \text{ in } A(n_+).$$

The $\Gamma$-set structure, multiplication, and unit are defined as follows:

\begin{enumerate}[label=(\roman*)]
\item \emph{Structure maps:} For $f \in \Hom_{\Gamma^{\mathrm{op}}}(n_+, m_+)$, we define:
$$(S^{-1}A)(f)\left(\frac{a}{s}\right) := \frac{A(f)(a)}{s}.$$

\item \emph{Multiplication:} For $\displaystyle\frac{a_1}{s_1} \in S^{-1}A(m_+)$ and $\displaystyle\frac{a_2}{s_2} \in S^{-1}A(n_+)$, we define their product in $S^{-1}A(m_+ \wedge n_+)$ by:
$$\frac{a_1}{s_1} \cdot \frac{a_2}{s_2} := \frac{a_1 \cdot a_2}{s_1 \cdot s_2}$$
where $a_1 \cdot a_2$ is the product in $A(m_+ \wedge n_+)$ via the $\mathbb{F}_1$-algebra structure on $A$, and $s_1 \cdot s_2$ is computed in $A(1_+)$.

\item \emph{Unit:} The unit element is $\displaystyle 1_{S^{-1}A} := \frac{1_A}{1_A} \in S^{-1}A(1_+)$.

\item \emph{Basepoint:} The basepoint at each level is $\displaystyle\frac{*}{1_A}$, where $* \in A(n_+)$ is the basepoint.
\end{enumerate}
\end{definition}

We verify that this construction yields a well-defined $\mathbb{F}_1$-algebra:

\begin{proposition}\label{prop:localization-welldef}
The localization $S^{-1}A$ defined in Definition \ref{def:localization} is a well-defined commutative $\mathbb{F}_1$-algebra.
\end{proposition}

\begin{proof}
It is straightforward to check that $\sim$ is an equivalence relation. We now verify that it is well defined with respect to $\Gamma$-set operations and the algebra structure.

Suppose $\displaystyle\frac{a_1}{s_1} \sim \frac{a_2}{s_2}$ in $S^{-1}A(n_+)$. Then there exists $t \in S$ with $ta_1s_2 = ta_2s_1$ in $A(n_+)$. For any $f \in \Hom_{\Gamma^{\mathrm{op}}}(n_+, m_+)$, applying $A(f)$ yields:
$$tA(f)(a_1)s_2 = A(f)(ta_1s_2) = A(f)(ta_2s_1) = tA(f)(a_2)s_1$$
by naturality of the action.

Therefore $\displaystyle\frac{A(f)(a_1)}{s_1} \sim \frac{A(f)(a_2)}{s_2}$ in $S^{-1}A(m_+)$, so $(S^{-1}A)(f)$ is well-defined.

For compatibility with the algebra structure of $S^{-1}A$, suppose $\displaystyle\frac{a_1}{s_1} \sim \frac{a_2}{s_2}$ in $S^{-1}A(m_+)$ and $\displaystyle\frac{b_1}{t_1} \sim \frac{b_2}{t_2}$ in $S^{-1}A(n_+)$. Then there exist $r, r' \in S$ such that:
$$ra_1s_2 = ra_2s_1 \quad\text{and}\quad r'b_1t_2 = r'b_2t_1.$$

Multiplying these equations using the $\mathbb{F}_1$-algebra structure, we get:
$$(rr')(a_1 \cdot b_1)(s_2 \cdot t_2) = (rr')(a_2 \cdot b_2)(s_1 \cdot t_1)$$
in $A(m_+ \wedge n_+)$. Since $rr' \in S$, we have:
$$\frac{a_1 \cdot b_1}{s_1 \cdot t_1} \sim \frac{a_2 \cdot b_2}{s_2 \cdot t_2}.$$

We then verify that $S^{-1}A$ is a commutative $\FF$-algebra.

\emph{Identity:} $(S^{-1}A)(\mathrm{id}_{n_+})\left(\displaystyle\frac{a}{s}\right) = \displaystyle\frac{A(\mathrm{id}_{n_+})(a)}{s} = \displaystyle\frac{a}{s}$ since $A$ is a functor.

\emph{Composition:} For $f: n_+ \to m_+$ and $g: m_+ \to k_+$, we have:
\begin{align*}
(S^{-1}A)(g \circ f)\left(\frac{a}{s}\right) &= \frac{A(g \circ f)(a)}{s}\\
&= \frac{A(g)(A(f)(a))}{s}\\
&= (S^{-1}A)(g)\left(\frac{A(f)(a)}{s}\right)\\
&= (S^{-1}A)(g) \circ (S^{-1}A)(f)\left(\frac{a}{s}\right).
\end{align*}

The associativity and unit axioms for multiplication follow from the corresponding axioms for $A$, using the fact that localization preserves the algebraic structure. The commutativity follows from the commutativity of $A$.
\end{proof}

\subsection{Universal property}

The localization $S^{-1}A$ satisfies a universal property characterizing it as the initial $\mathbb{F}_1$-algebra in which elements of $S$ become units:

\begin{theorem}[Theorem \ref{thm:main-localization}]\label{thm:localization-universal}
Let $A$ be a commutative $\mathbb{F}_1$-algebra and $S \subseteq A(1_+)$ a multiplicatively closed subset, then there exists a natural morphism of $\mathbb{F}_1$-algebras:
$$\iota: A \to S^{-1}A, \quad a \mapsto \frac{a}{1_A}$$
such that $\iota(s)$ is a unit in $(S^{-1}A)(1_+)$ for all $s \in S$.

Moreover, this morphism is universal: for any commutative $\mathbb{F}_1$-algebra $B$ and morphism $\varphi: A \to B$ such that $\varphi(s) \in B(1_+)^\times$ for all $s \in S$, there exists a unique morphism $\overline{\varphi}: S^{-1}A \to B$ such that $\overline{\varphi} \circ \iota = \varphi$.
\end{theorem}

\begin{proof}
The map $a \mapsto \displaystyle\frac{a}{1_A}$ is clearly a well-defined $\FF$-algebra map, and that $\iota(s)$ is a unit in $(S^{-1}A)(1_+)$ for all $s \in S$.

Now given $\varphi: A \to B$ with $\varphi(s) \in B(1_+)^\times$ for all $s \in S$, define $\overline{\varphi}: S^{-1}A \to B$ at level $n_+$ by:
$$\overline{\varphi}\left(\frac{a}{s}\right) := \varphi(a) \cdot \varphi(s)^{-1}$$
where $\varphi(s)^{-1}$ denotes the multiplicative inverse in $B(1_+)$.

\emph{Well-definedness:} If $\displaystyle\frac{a_1}{s_1} = \frac{a_2}{s_2}$ in $S^{-1}A(n_+)$, then there exists $t \in S$ such that $ta_1s_2 = ta_2s_1$ in $A(n_+)$. Applying $\varphi$:
$$\varphi(t) \cdot \varphi(a_1) \cdot \varphi(s_2) = \varphi(t) \cdot \varphi(a_2) \cdot \varphi(s_1).$$

Since $\varphi(t), \varphi(s_1), \varphi(s_2) \in B(1_+)^\times$, we can multiply both sides by $\varphi(t)^{-1} \varphi(s_1)^{-1} \varphi(s_2)^{-1}$ to obtain:
$$\varphi(a_1) \cdot \varphi(s_1)^{-1} = \varphi(a_2) \cdot \varphi(s_2)^{-1}.$$

\emph{$\overline{\varphi}$ is a morphism of $\Gamma$-sets:} For $f \in \Hom_{\Gamma^{\mathrm{op}}}(n_+, m_+)$:
\begin{align*}
\overline{\varphi}\left((S^{-1}A)(f)\left(\frac{a}{s}\right)\right) &= \overline{\varphi}\left(\frac{A(f)(a)}{s}\right)\\
&= \varphi(A(f)(a)) \cdot \varphi(s)^{-1}\\
&= B(f)(\varphi(a)) \cdot \varphi(s)^{-1}\\
&= B(f)(\varphi(a) \cdot \varphi(s)^{-1})\\
&= B(f)\left(\overline{\varphi}\left(\frac{a}{s}\right)\right)
\end{align*}
where the third equality uses that $\varphi(s)^{-1} \in B(1_+)$ and $B(f)$ is natural, and the fourth equality uses that elements of $B(1_+)$ act on higher levels in a compatible way.

$\overline{\varphi}$ preserves multiplication and unit by direct computation and that
$$\overline{\varphi}(\iota(a)) = \overline{\varphi}\left(\frac{a}{1_A}\right) = \varphi(a) \cdot \varphi(1_A)^{-1} = \varphi(a) \cdot 1_B^{-1} = \varphi(a).$$

Now suppose $\psi: S^{-1}A \to B$ is another morphism with $\psi \circ \iota = \varphi$. For any $\displaystyle\frac{a}{s} \in S^{-1}A(n_+)$, we have:
$$\frac{a}{s} = \frac{a}{1_A} \cdot \frac{1_A}{s} = \iota(a) \cdot \iota(s)^{-1}$$
in $S^{-1}A(n_+)$. Therefore:
$$\psi\left(\frac{a}{s}\right) = \psi(\iota(a)) \cdot \psi(\iota(s))^{-1} = \varphi(a) \cdot \varphi(s)^{-1} = \overline{\varphi}\left(\frac{a}{s}\right). $$
\end{proof}

\begin{remark}[Alternative localization]\label{rem:alternative-localization}
One might consider denominators from arbitrary levels of
$A$. More precisely, a \emph{graded multiplicatively closed set} would be a collection $S = \{S(n_+)\}_{n \geq 0}$ where $S(n_+) \subseteq A(n_+)$, closed under the multiplication maps $A(m_+) \wedge A(n_+) \to A(m_+ \wedge n_+)$ and $1_A \in S(1_+)$. One could then define a localization by
$$S_G^{-1}A(n_+) := \left\{\frac{a}{s} \,\bigg|\, a \in A(m_+ \wedge n_+),\, s \in S(m_+) \text{ for some } m \geq 1\right\} \bigg/ \sim$$
with an appropriately defined equivalence relation.

However, any denominator of level greater than one
forces the resulting algebra to be zero. Indeed, if some $s \in A(2_+)$ is to be inverted (and similarly for higher levels), consider $x=(1,0)/s, y=(0,1)/s$, where $(1,0), (0,1)$ stands for $A(\iota_1)(1_A), A(\iota_2)(1_A)$ respectively, and $\iota_j:1_+\to 2_+$ sends $1$ to $j$. Since $sx=(1,0), sy=(0,1)$, taking the first and second projection one gets
$$s_1x=1, s_2x=0,s_1y=0, s_2y=1.$$
where $s_i:=A(p_i)(s)$, and $p_i:2_+\to1_+$ is the $i$th coordinate projection.
The first equality makes $x$ a unit in $S_G^{-1}A$, so the second
implies $s_2=0$, contradicting the fourth unless $0=1$. Consequently, the localization is the zero algebra.

\end{remark}

\subsection{Special cases and examples}

\begin{example}[Localization at a prime]\label{ex:localization-at-prime}
Let $\mathfrak{p} \subseteq A(1_+)$ be a prime ideal. The localization of $A$ at $\mathfrak{p}$ is:
$$A_\mathfrak{p} := S^{-1}A$$
where $S = A(1_+) \setminus \mathfrak{p}$ is the complement of $\mathfrak{p}$.

At level $n_+$, we have:
$$A_\mathfrak{p}(n_+) = \left\{\frac{a}{s} \,\bigg|\, a \in A(n_+), s \in A(1_+) \setminus \mathfrak{p}\right\} \bigg/ \sim.$$
\end{example}

\begin{example}[Localization at an element]\label{ex:localization-at-element}
For $f \in A(1_+)$, the localization at $f$ is:
$$A_f := S^{-1}A$$
where $S = \{1_A, f, f^2, f^3, \ldots\}$.

At level $n_+$:
$$A_f(n_+) = \left\{\frac{a}{f^k} \,\bigg|\, a \in A(n_+), k \geq 0\right\} \bigg/ \sim.$$
\end{example}

\begin{example}[Localization of Eilenberg-MacLane algebras]\label{ex:EM-localization}
Let $R$ be a commutative ring and $S \subseteq R$ a multiplicatively closed subset (not containing $0$). Then:
$$S^{-1}(HR) \cong H(S^{-1}R)$$
where the right side is the Eilenberg-MacLane algebra of the classical localization ring.

\begin{proof}
At level $n_+$, we have:
$$S^{-1}(HR)(n_+) = \left\{\frac{(r_1, \ldots, r_n)}{s} \,\bigg|\, r_i \in R, s \in S\right\} \bigg/ \sim$$
$$= \left\{\left(\frac{r_1}{s}, \ldots, \frac{r_n}{s}\right) \,\bigg|\, r_i \in R, s \in S\right\} \bigg/ \sim'$$
$$\cong (S^{-1}R)^{\oplus n} = H(S^{-1}R)(n_+)$$
where $\sim'$ is the componentwise equivalence relation in $S^{-1}R$.

The isomorphism is given by:
$$\frac{(r_1, \ldots, r_n)}{s} \mapsto \left(\frac{r_1}{s}, \ldots, \frac{r_n}{s}\right).$$

One verifies that this respects the $\Gamma$-set structure and the $\mathbb{F}_1$-algebra structure.
\end{proof}
\end{example}

\begin{example}[Localization of spherical algebras]\label{ex:spherical-localization}
Let $M$ be a pointed commutative monoid and $S \subseteq M \setminus \{0\}$ a multiplicatively closed subset. Then:
$$S^{-1}(\mathbb{S}M) \cong \mathbb{S}(S^{-1}M)$$
where $S^{-1}M$ is Deitmar's localization of the monoid $M$.

\begin{proof}
At level $n_+$:
$$S^{-1}(\mathbb{S}M)(n_+) = S^{-1}(M \wedge n_+) \cong (S^{-1}M) \wedge n_+ = \mathbb{S}(S^{-1}M)(n_+).$$

The isomorphism is:
$$\frac{(m, i)}{s} \mapsto \left(\frac{m}{s}, i\right)$$
where $(m, i) \in M \wedge n_+$ represents the element $m$ in the $i$-th position.
\end{proof}
\end{example}

\section{Sheaves of $\mathbb F_1$-algebras}
\label{sec:sheaves}

We establish the sheaf-theoretic constructions needed
for our spectrum. Limits and filtered colimits of
commutative $\mathbb F_1$-algebras are computed levelwise
on the underlying pointed sets. Consequently, the sheaf
condition can be checked levelwise, and stalks inherit
their algebra structures from filtered colimits.
Sheafification also inherits an algebra structure,
using the fact that sheafification of pointed sets
preserves finite products.

\subsection{Presheaves and sheaves}

\begin{definition}\label{def:presheaf-f1alg}
Let $X$ be a topological space. A \emph{presheaf of
commutative $\mathbb F_1$-algebras} on $X$ is a functor
\[
\mathcal F:\operatorname{Open}(X)^{\mathrm{op}}
\longrightarrow\FAlg,
\]
where $\operatorname{Open}(X)$ denotes the category
of open subsets of $X$, with inclusions as morphisms.
Morphisms of presheaves are natural transformations.
\end{definition}

For each $n\geq0$, evaluation at $n_+$ gives a
presheaf of pointed sets
\[
\mathcal F_n(U):=\mathcal F(U)(n_+).
\]
The structure maps, multiplication, and units of
$\mathcal F(U)$ are compatible with restriction,
and hence define operations on these presheaves.

\begin{definition}\label{def:sheaf-f1alg}
A presheaf $\mathcal F$ of commutative
$\mathbb F_1$-algebras is a \emph{sheaf} if, for
every open covering $U=\bigcup_{i\in I}U_i$, the
diagram
\[
\mathcal F(U)
\longrightarrow
\prod_{i\in I}\mathcal F(U_i)
\rightrightarrows
\prod_{(i,j)\in I\times I}
\mathcal F(U_i\cap U_j)
\]
is an equalizer in $\FAlg$. The two parallel arrows
are induced by restriction from $U_i$ and $U_j$,
respectively.
\end{definition}

The empty covering is included in this definition.
In particular, $\mathcal F(\varnothing)$ is the
terminal object of $\FAlg$, namely the zero
$\mathbb F_1$-algebra.

\begin{lemma}\label{lem:sheaf-levelwise}
A presheaf $\mathcal F$ of commutative
$\mathbb F_1$-algebras is a sheaf if and only if
$\mathcal F_n$ is a sheaf of pointed sets for every
$n\geq0$.
\end{lemma}

\begin{proof}
Products and equalizers in $\FAlg$ are computed
levelwise on the underlying pointed sets.
Thus the equalizer condition in
Definition~\ref{def:sheaf-f1alg} is equivalent
to the sheaf condition at every level.
\end{proof}

\begin{notation}
We write $\mathrm{PSh}_{\mathbb F_1}(X)$ for the
category of presheaves of commutative
$\mathbb F_1$-algebras on $X$, and
$\mathrm{Sh}_{\mathbb F_1}(X)$ for its full
subcategory of sheaves.
\end{notation}

\subsection{Sheafification}

\begin{definition}[Sheafification]
\label{cons:sheafification}
Let $\mathcal F$ be a presheaf of commutative
$\mathbb F_1$-algebras on $X$. Define
$\mathcal F^\#$ levelwise by
\[
\mathcal F^\#(U)(n_+)
:=
(\mathcal F_n)^\#(U),
\]
where the right-hand side is the ordinary
sheafification of a presheaf of pointed sets.
The $\Gamma$-set structure maps are induced by
functoriality of sheafification.
\end{definition}

\begin{theorem}\label{lem:sheafification-f1alg}
The levelwise sheafification $\mathcal F^\#$
has a canonical structure of a sheaf of commutative
$\mathbb F_1$-algebras. The canonical morphism
\[
\eta:\mathcal F\longrightarrow\mathcal F^\#
\]
is universal among morphisms from $\mathcal F$
to sheaves of commutative $\mathbb F_1$-algebras.
\end{theorem}

\begin{proof}
Functoriality of sheafification gives the
$\Gamma$-set structure maps and ensures that
they preserve identities and composition.
Since $\mathcal F_0$ is the singleton presheaf,
its sheafification is again the singleton sheaf.
Thus $\mathcal F^\#(U)(0_+)=\{*\}$ for every $U$.

For finite pointed sets $m_+$ and $n_+$,
multiplication in $\mathcal F$ gives a morphism
of presheaves of pointed sets
\[
\mathcal F_m\times\mathcal F_n
\longrightarrow
\mathcal F_{mn},
\]
where we identify $m_+\wedge n_+$ with $(mn)_+$.
Sheafification preserves finite products, so
applying it yields
\[
\mathcal F_m^\#\times\mathcal F_n^\#
\longrightarrow
\mathcal F_{mn}^\#.
\]
This multiplication sends a pair to the basepoint
whenever either entry is the basepoint.
It therefore factors, on sections, through the
smash product. Naturality in $m_+$ and $n_+$
gives a multiplication
\[
\mathcal F^\#(U)\wedge\mathcal F^\#(U)
\longrightarrow\mathcal F^\#(U)
\]
for the Day convolution product.

Composing the unit maps
$\FF\to\mathcal F(U)$ with the canonical maps
$\mathcal F(U)\to\mathcal F^\#(U)$ gives the
units of $\mathcal F^\#(U)$.
Associativity, commutativity, and the unit axioms
can be checked locally on representatives from
$\mathcal F$, where they already hold.
The same construction shows that restriction
maps and $\eta$ preserve these operations.
By Lemma~\ref{lem:sheaf-levelwise},
$\mathcal F^\#$ is a sheaf.

Now let $\mathcal G$ be a sheaf of commutative
$\mathbb F_1$-algebras and let
$\varphi:\mathcal F\to\mathcal G$ be a morphism
of presheaves. At each level, the universal
property of sheafification gives a unique map
\[
\psi_n:\mathcal F_n^\#\longrightarrow\mathcal G_n
\]
factoring $\varphi_n$.
Uniqueness makes these maps compatible with the
$\Gamma$-set structure maps.
They preserve multiplication and units because
these identities can be checked locally on
representatives from $\mathcal F$, where they
follow from the corresponding properties of
$\varphi$.
Thus the maps $\psi_n$ assemble into a unique
morphism
\[
\psi:\mathcal F^\#\longrightarrow\mathcal G
\]
such that $\psi\circ\eta=\varphi$.
\end{proof}

\begin{corollary}\label{cor:sheafification-adjunction}
Sheafification is left adjoint to the inclusion
\[
i:\mathrm{Sh}_{\mathbb F_1}(X)
\hookrightarrow
\mathrm{PSh}_{\mathbb F_1}(X).
\]
Equivalently, there is a natural bijection
\[
\Hom_{\mathrm{Sh}_{\mathbb F_1}(X)}
(\mathcal F^\#,\mathcal G)
\cong
\Hom_{\mathrm{PSh}_{\mathbb F_1}(X)}
(\mathcal F,i\mathcal G).
\]
\end{corollary}

\subsection{Direct and inverse images}

Let $f:X\to Y$ be a continuous map.

\begin{definition}\label{def:direct-image}
For a sheaf $\mathcal F$ of commutative
$\mathbb F_1$-algebras on $X$, its
\emph{direct image} is the sheaf on $Y$ defined by
\[
(f_*\mathcal F)(V)
=
\mathcal F(f^{-1}(V)).
\]
\end{definition}

The sheaf condition follows levelwise from the
usual direct image construction.

\begin{definition}\label{def:inverse-image}
For a sheaf $\mathcal G$ of commutative
$\mathbb F_1$-algebras on $Y$, define a presheaf
on $X$ by
\[
(f_p\mathcal G)(U)
=
\colim_{\substack{V\subseteq Y\text{ open}\\
                  f(U)\subseteq V}}
\mathcal G(V),
\]
with transition maps given by restriction.
The \emph{inverse image} of $\mathcal G$ is
\[
f^{-1}\mathcal G:=(f_p\mathcal G)^\#.
\]
\end{definition}

The indexing category in this colimit is filtered:
any two open sets containing $f(U)$ have an
intersection containing $f(U)$.
Thus the colimit is computed levelwise and
inherits a commutative $\mathbb F_1$-algebra
structure.

\begin{theorem}\label{thm:adjunction-direct-inverse}
The inverse and direct image functors form an
adjoint pair
\[
f^{-1}\dashv f_*.
\]
In particular, there is a natural bijection
\[
\Hom_{\mathrm{Sh}_{\mathbb F_1}(X)}
(f^{-1}\mathcal G,\mathcal F)
\cong
\Hom_{\mathrm{Sh}_{\mathbb F_1}(Y)}
(\mathcal G,f_*\mathcal F).
\]
\end{theorem}

\begin{proof}
The ordinary inverse image--direct image
adjunction applies at each level.
Its bijections are natural with respect to
the $\Gamma$-set structure maps and compatible
with multiplication and units.
They therefore restrict to the stated bijection
for morphisms of sheaves of
$\mathbb F_1$-algebras.
\end{proof}

\subsection{Stalks}

\begin{definition}\label{def:stalk}
Let $\mathcal F$ be a presheaf of commutative
$\mathbb F_1$-algebras on $X$, and let $x\in X$.
Its \emph{stalk} at $x$ is
\[
\mathcal F_x
=
\colim_{x\in U}\mathcal F(U),
\]
where $U$ ranges over the open neighborhoods
of $x$, with transition maps given by restriction.
\end{definition}

Since this colimit is filtered,
\[
\mathcal F_x(n_+)
\cong
(\mathcal F_n)_x.
\]
Thus elements at level $n_+$ of the stalk are
germs of sections at that level, and all
operations are induced from local representatives.

When $\mathcal F$ is a sheaf and
$i:\{x\}\hookrightarrow X$ is the inclusion,
there is a canonical identification
\[
\mathcal F_x\cong(i^{-1}\mathcal F)(\{x\}).
\]

\begin{lemma}\label{lem:stalk-sheafification}
For every presheaf $\mathcal F$ of commutative
$\mathbb F_1$-algebras and every $x\in X$,
the sheafification morphism induces an isomorphism
\[
\mathcal F_x\xrightarrow{\sim}\mathcal F^\#_x.
\]
\end{lemma}

\begin{proof}
Sheafification of pointed sets preserves stalks.
The resulting levelwise isomorphisms are induced
by the morphism $\mathcal F\to\mathcal F^\#$,
and hence respect the $\mathbb F_1$-algebra
structures.
\end{proof}

\begin{lemma}\label{lem:stalk-pullback}
Let $f:X\to Y$ be continuous and let $\mathcal G$
be a sheaf of commutative $\mathbb F_1$-algebras
on $Y$. For every $x\in X$, there is a canonical
isomorphism
\[
(f^{-1}\mathcal G)_x\cong\mathcal G_{f(x)}.
\]
\end{lemma}

\begin{proof}
By Lemma \ref{lem:stalk-sheafification} and the definition $f^{-1}\mathcal{G} = (f_p\mathcal{G})^\#$, we have:
\begin{align*}
(f^{-1}\mathcal{G})_x &\cong (f_p\mathcal{G})_x \\
&= \colim_{x \in U} f_p\mathcal{G}(U) \\
&= \colim_{x \in U} \colim_{f(U) \subseteq V} \mathcal{G}(V) \\
&\cong \colim_{f(x) \in V} \mathcal{G}(V) \\
&= \mathcal{G}_{f(x)}
\end{align*}
where the isomorphism on the third line follows from the fact that the composition of filtered colimits (over neighborhoods of $x$ and neighborhoods of $f(x)$ containing $f(U)$) is isomorphic to the filtered colimit over all neighborhoods of $f(x)$.
\end{proof}

\subsection{Morphisms of sheaves}

\begin{definition}\label{def:f-map}
Let $f:X\to Y$ be continuous, and let
$\mathcal F$ and $\mathcal G$ be sheaves of
commutative $\mathbb F_1$-algebras on $X$ and
$Y$, respectively. An \emph{$f$-map}
from $\mathcal G$ to $\mathcal F$ is a morphism
\[
\phi:\mathcal G\longrightarrow f_*\mathcal F.
\]
By Theorem~\ref{thm:adjunction-direct-inverse},
it is equivalently a morphism
\[
\widetilde\phi:f^{-1}\mathcal G
\longrightarrow\mathcal F.
\]
\end{definition}

\begin{lemma}\label{lem:stalk-map}
An $f$-map $\phi:\mathcal G\to f_*\mathcal F$
induces, for every $x\in X$, a morphism of
commutative $\mathbb F_1$-algebras
\[
\phi_x:\mathcal G_{f(x)}
\longrightarrow\mathcal F_x.
\]
For every open neighborhood $V$ of $f(x)$,
the diagram
\[
\begin{tikzcd}
\mathcal G(V)
  \arrow[r]
  \arrow[d,"\phi_V"']
&
\mathcal G_{f(x)}
  \arrow[d,"\phi_x"]
\\
\mathcal F(f^{-1}(V))
  \arrow[r]
&
\mathcal F_x
\end{tikzcd}
\]
commutes, where the horizontal arrows send
sections to their germs.
\end{lemma}

\begin{proof}
Take the stalk at $x$ of the adjoint morphism
$\widetilde\phi:f^{-1}\mathcal G\to\mathcal F$
and use Lemma~\ref{lem:stalk-pullback}.
Equivalently, $\phi_x$ sends the germ of a
section $s\in\mathcal G(V)(n_+)$ at $f(x)$
to the germ of $\phi_V(s)$ at $x$.
This description gives the stated diagram.
\end{proof}

\section{The prime spectrum}\label{sec:spectrum}

We associate to a commutative $\mathbb F_1$-algebra $A$
the spectrum of its underlying multiplicative monoid,
equipped with a sheaf retaining the full algebra
structure. We prove that its stalks and sections over
principal open sets are given by localization, and
that its global sections recover $A$.
We then introduce the corresponding geometric objects
and compare our construction with that of
Connes and Consani.

\subsection{The spectrum and its structure sheaf}

\begin{definition}\label{def:prime-spectrum}
Let $A$ be a commutative $\mathbb F_1$-algebra.
Set
\[
X_A:=|\Spec A|
:=\Spec_{\mathrm{Deitmar}}A(1_+).
\]
For $\mathfrak p\in X_A$, write
\[
A_{\mathfrak p}
:=
(A(1_+)\setminus\mathfrak p)^{-1}A.
\]

For an open subset $U\subseteq X_A$, define
$\mathcal O_A(U)(n_+)$ to be the pointed set of
functions
\[
s:U\longrightarrow
\coprod_{\mathfrak p\in U}A_{\mathfrak p}(n_+)
\]
such that $s(\mathfrak p)\in A_{\mathfrak p}(n_+)$
and $s$ is locally a quotient.

Explicitly, for every $\mathfrak p\in U$,
there must exist an open neighborhood
$V\subseteq U$ of $\mathfrak p$ and elements
$a\in A(n_+)$ and $f\in A(1_+)$ such that
$V\subseteq D(f)$ and
\[
s(\mathfrak q)=\frac{a}{f}
\quad\text{in }A_{\mathfrak q}(n_+)
\qquad(\mathfrak q\in V).
\]
The basepoint, $\Gamma$-set structure maps,
multiplication, and units are defined pointwise.
Restriction maps are given by restriction of
functions.
\end{definition}

These operations preserve the local quotient
condition. Indeed, on a common neighborhood where
two sections are represented by $a/f$ and $b/g$,
their product is represented by $ab/(fg)$.
For a pointed map $u:n_+\to k_+$, the corresponding
structure map sends a section represented by $a/f$
to one represented by $A(u)(a)/f$.
The algebra axioms hold pointwise in the
localizations $A_{\mathfrak p}$.

Compatible local sections glue uniquely as
functions, and the local quotient condition
is preserved under this gluing.
Thus each level is a sheaf of pointed sets.
By Lemma~\ref{lem:sheaf-levelwise},
$\mathcal O_A$ is a sheaf of commutative
$\mathbb F_1$-algebras.

\begin{definition}\label{def:affine}
The \emph{prime spectrum} of $A$ is the pair
\[
\Spec A:=(X_A,\mathcal O_A).
\]
We call $\mathcal O_A$ its \emph{structure sheaf}.
\end{definition}

\begin{remark}\label{rem:levels-of-structure}
At the first level, $\mathcal O_A$ is precisely
the monoid structure sheaf of $A(1_+)$:
\[
\mathcal O_A(-)(1_+)
\cong
\mathcal O_{A(1_+)}.
\]
The higher levels retain the additional
$\Gamma$-set structure of $A$.

For a spherical algebra $A=\mathbb S M$,
localization satisfies
\[
A_{\mathfrak p}\cong\mathbb S(M_{\mathfrak p}).
\]
Thus the construction recovers affine monoid
geometry with its local algebraic data expressed
through the spherical functor.
For general $\mathbb F_1$-algebras, the higher
levels carry additional information.
\end{remark}

\subsection{Stalks and sections}

\begin{theorem}[Theorem~\ref{thm:main-structure-sheaf}]
\label{thm:stalks-and-sections}
Let $A$ be a commutative $\mathbb F_1$-algebra.
\begin{enumerate}[label=(\roman*)]
\item
For every $\mathfrak p\in X_A$, there is a
canonical isomorphism
\[
\mathcal O_{A,\mathfrak p}\cong A_{\mathfrak p}.
\]

\item
For every $f\in A(1_+)$, there is a canonical
isomorphism
\[
A_f\xrightarrow{\sim}\mathcal O_A(D(f)).
\]

\item
Global sections recover $A$:
\[
\Gamma(X_A,\mathcal O_A)\cong A.
\]
\end{enumerate}
All isomorphisms are isomorphisms of
$\mathbb F_1$-algebras.
\end{theorem}

\begin{proof}
\begin{enumerate}[label=(\roman*)]
\item
Evaluation at $\mathfrak p$ defines a morphism
\[
\varphi:\mathcal O_{A,\mathfrak p}
\longrightarrow A_{\mathfrak p},
\qquad
[s]_{\mathfrak p}\longmapsto s(\mathfrak p).
\]
We verify that it is bijective at every level.

For surjectivity, let $a/g\in A_{\mathfrak p}(n_+)$,
where $g\notin\mathfrak p$.
The section on $D(g)$ defined by
\[
\mathfrak q\longmapsto
\frac{a}{g}\in A_{\mathfrak q}(n_+)
\]
has a germ at $\mathfrak p$ mapping to $a/g$.

For injectivity, suppose two germs have the same
image. After restricting their representatives
to a common neighborhood $U$ of $\mathfrak p$,
we may write them as $a/g$ and $b/h$, with
$U\subseteq D(g)\cap D(h)$.
Equality in $A_{\mathfrak p}(n_+)$ means that
there exists $k\notin\mathfrak p$ such that
\[
kha=kgb.
\]
The two sections therefore agree on
$U\cap D(k)$ and define the same germ.

Since evaluation respects all operations,
$\varphi$ is an isomorphism of
$\mathbb F_1$-algebras.

\item If $D(f)=\varnothing$, then $f$ is nilpotent in $A(1_+)$, and $A_f$ is the zero $\mathbb F_1$-algebra, as is $\mathcal O_A(\varnothing)$. 
        
        We now assume that $D(f)$ is nonempty. Define a morphism $\psi$ of $\mathbb{F}_1$-algebras:
        \begin{align*}
            \psi : A_f &\rightarrow \mathcal{O}_A(D(f))\\
            \psi(k_+):\displaystyle \frac{a}{f^n} & \mapsto s
        \end{align*}
        where $s(p)=\displaystyle\frac{a}{f^n} \in A_p(k_+)$ for all $p\in D(f)$. 
        
        \emph{Injectivity:} Suppose $\psi(\displaystyle \frac{a}{f^n})=\psi(\displaystyle \frac{b}{f^m})$. Then for every $p \in D(f)$, we have $\displaystyle \frac{a}{f^n}=\displaystyle \frac{b}{f^m}$ in $A_p$. This implies that for each $p \in D(f)$, there exists $h \in A(1_+) \setminus p$ such that $h_pf^ma=h_pf^nb$.

        Consider the ideal:
$$I := \{h \in A(1_+) \mid hf^ma = hf^nb \text{ in } A(k_+)\}.$$

This is an ideal of $A(1_+)$, and that for each $\mathfrak{p} \in D(f)$, we have $h_\mathfrak{p} \in I \setminus \mathfrak{p}$, so $I \not\subseteq \mathfrak{p}$. Therefore:
$$V(I) \cap D(f) = \emptyset$$
which implies $V(I) \subseteq V(f)$.

By Proposition \ref{thm:monoid-nullstellensatz}(ii), $f \in \sqrt{I}$, so $f^\ell \in I$ for some $\ell \geq 1$. Thus:
$$f^\ell \cdot f^m a = f^\ell \cdot f^n b$$
which means $\displaystyle\frac{a}{f^n} = \frac{b}{f^m}$ in $A_f(k_+)$.

        \emph{Surjectivity:} Let $s \in \mathcal{O}_A(D(f))(k_+)$. We must show that $s = \psi\left(\displaystyle\frac{a}{f^\ell}\right)$ for some $a \in A(k_+)$ and $\ell \geq 0$.

Consider the ideal $\mathfrak{c} := \bigcup_{\mathfrak{p} \in D(f)} \mathfrak{p}$. This is the unique largest prime ideal of $D(f)$ under inclusion, since arbitrary union of monoid ideals is again an ideal and the fact that $\mathfrak{c}$ is prime. 

$D(f)$ is the smallest open neighborhood of $\mathfrak c$. Indeed, if a basic open $D(g)$ contains $\mathfrak c$, then
\[
g\notin\mathfrak c.
\]
By definition of $\mathfrak c$, this means that $g$ belongs to no prime $\mathfrak p\in D(f)$, hence
\[
D(f)\subseteq D(g).
\]
Since basic opens form a basis, every open neighborhood of $\mathfrak c$ contains $D(f)$.

Since $s$ is locally a quotient at $\mathfrak{c}$, there exist $a \in A(k_+)$ and $g \in A(1_+)$ with $D(f) \subseteq D(g)$ such that $s(\mathfrak{p}) = \displaystyle\frac{a}{g}$ for all $\mathfrak{p} \in D(f)$.

Now, $D(f) \subseteq D(g)$ means $V(g) \subseteq V(f)$, so by Proposition \ref{thm:monoid-nullstellensatz}(ii), $f \in \sqrt{\langle g \rangle}$. Thus $f^n = bg$ for some $b \in A(1_+)$ and $n \geq 1$.

Therefore:
$$\frac{a}{g} = \frac{ab}{bg} = \frac{ab}{f^n} \in A_f(k_+).$$

Thus $s = \psi\left(\displaystyle\frac{ab}{f^n}\right)$, establishing surjectivity.

\item
Taking $f=1_A$ in part~(ii) gives
\[
\Gamma(X_A,\mathcal O_A)
\cong A_{1_A}\cong A.
\]
\end{enumerate}
\end{proof}

\subsection{Locally $\mathbb F_1\mathrm{Alg}$-ed spaces}

We now specify the geometric category in which
the spectrum construction takes values.

\begin{definition}\label{def:locally-f1alg-space}
A \emph{locally $\mathbb F_1\mathrm{Alg}$-ed space}
is a pair $(X,\mathcal O_X)$ consisting of a
topological space $X$ and a sheaf of commutative
$\mathbb F_1$-algebras on $X$.

A morphism
\[
(f,f^\#):(X,\mathcal O_X)\longrightarrow
(Y,\mathcal O_Y)
\]
consists of a continuous map $f:X\to Y$ and
a morphism of sheaves
\[
f^\#:\mathcal O_Y\longrightarrow f_*\mathcal O_X
\]
such that every induced stalk map
\[
f^\#_x:\mathcal O_{Y,f(x)}
\longrightarrow\mathcal O_{X,x}
\]
is \emph{local at the first level}, meaning
\[
\bigl(f^\#_x(1_+)\bigr)^{-1}
\bigl(\mathcal O_{X,x}(1_+)^\times\bigr)
=
\mathcal O_{Y,f(x)}(1_+)^\times.
\]
\end{definition}

\begin{remark}\label{rem:locality-condition}
Every morphism of $\mathbb F_1$-algebras preserves
units at the first level. The locality condition
requires it also to reflect units, or equivalently
to send non-units to non-units.
This is the monoid-theoretic locality condition
applied to the first level of each stalk.
\end{remark}

\begin{definition}\label{def:absolute-affine-scheme}
An \emph{absolute affine scheme} is a locally
$\mathbb F_1\mathrm{Alg}$-ed space isomorphic to
$\Spec A$ for some commutative
$\mathbb F_1$-algebra $A$.
We denote their full subcategory by
$\AffSch_{\mathbb F_1}$.
\end{definition}

\begin{proposition}\label{prop:basic-opens-basis}
Let $A$ be a commutative $\mathbb F_1$-algebra.
For every $f\in A(1_+)$, there is a canonical
isomorphism of locally
$\mathbb F_1\mathrm{Alg}$-ed spaces
\[
\Spec(A_f)
\xrightarrow{\sim}
(D(f),\mathcal O_A|_{D(f)}).
\]
In particular, the principal open sets form a
basis of absolute affine open subsets.

Moreover, if $D(g)\subseteq D(f)$, then the
canonical identifications of sections commute
with restriction:
\[
\begin{tikzcd}
A_f \arrow[r,"\sim"] \arrow[d]
&
\mathcal O_A(D(f)) \arrow[d]
\\
A_g \arrow[r,"\sim"]
&
\mathcal O_A(D(g)).
\end{tikzcd}
\]
The left vertical map is induced by localization.
\end{proposition}

\begin{proof}
Write $M=A(1_+)$.
Since $A_f(1_+)\cong M_f$, the prime ideal
correspondence for monoid localization gives
a homeomorphism
\[
\theta:
\Spec_{\mathrm{Deitmar}}M_f
\xrightarrow{\sim}D(f)
\]
by contraction.
A principal open defined by $a/f^r\in M_f$
corresponds under $\theta$ to
\[
D(f)\cap D(a)=D(fa).
\]

The universal property of localization gives
canonical isomorphisms
\[
(A_f)_{a/f^r}\cong A_{fa}.
\]
By Theorem~\ref{thm:stalks-and-sections},
these are precisely the section algebras of
the two structure sheaves on the corresponding
principal opens.

These identifications are compatible with
restriction: they are induced by the canonical
localization maps and agree on fractions.
They therefore give an isomorphism of the
structure sheaves under $\theta$.
The resulting stalk maps are isomorphisms,
hence local at the first level.
This proves the first assertion, including
the empty case.

Finally, if $D(g)\subseteq D(f)$, then
$g^N=bf$ for some $N\geq1$ and $b\in M$.
Thus $f$ becomes invertible in $A_g$, giving
the canonical map $A_f\to A_g$.
Both paths in the displayed diagram send
$a/f^r$ to the section represented by the
same fraction on $D(g)$.
\end{proof}

\begin{definition}
An \emph{absolute scheme} is a locally
$\mathbb F_1\mathrm{Alg}$-ed space
$(X,\mathcal O_X)$ such that every point has
an open neighborhood $U$ for which
$(U,\mathcal O_X|_U)$ is an absolute affine
scheme.
\end{definition}

\begin{remark}\label{rem:prime}
One could instead consider multiplicative prime
ideals that are also closed under the additive
operations encoded by $A$.
At the binary level, define
\[
a\oplus b
=
\left\{
A(s)(z)
\;\middle|\;
z\in A(2_+),\
A(p_1)(z)=a,\
A(p_2)(z)=b
\right\},
\]
where $p_i:2_+\to1_+$ sends $i$ to $1$ and
the other nonzero element to $0$, while
$s:2_+\to1_+$ sends both nonzero elements
to $1$.
Additive closure of an ideal $I$ means
\[
a,b\in I\quad\Longrightarrow\quad a\oplus b\subseteq I.
\]
For $A=HR$, this selects the ordinary ring
prime ideals. For $A=\mathbb S M$, it imposes
no additional restriction on monoid ideals.

For general $\mathbb F_1$-algebras, however,
this condition can leave a nonzero algebra
with no prime ideals.
To see this, set
\[
J(n_+)=(2\ZZ)^n\cup(3\ZZ)^n
\subseteq H\ZZ(n_+).
\]
The subsets $J(n_+)$ are stable under all
$\Gamma$-structure maps, and multiplication
satisfies
\[
J(m_+)\cdot H\ZZ(n_+)
\subseteq J(m_+\wedge n_+).
\]
Thus $J$ is a $\Gamma$-stable multiplicative
ideal, and collapsing it to the basepoint
at each level defines a commutative
$\mathbb F_1$-algebra
\[
A:=H\ZZ/J,
\qquad
A(n_+)=\ZZ^n/J(n_+).
\]
Here the quotient is a quotient of pointed
sets at each level, rather than a quotient
of rings. In particular, $A$ is nonzero,
since $1\notin J(1_+)$.

Now we have 
    $$0\oplus 0=A(1_+)$$
    since any $(2x,3y) \in A(2_+)$ could exhibit a sum, $\forall x, y \in \mathbb{Z}$, and $2,3$ are coprime. Therefore, a single $0 \in I$ indicates that the entire $A(1_+) \subseteq I$, contradicting the requirement that a prime ideal be proper.

\end{remark}

\subsection{Comparison with the Connes--Consani construction}
\label{subsec:comparison-cc}

The decision to omit additive closure from the definition of prime
ideals, discussed in the preceding remark, also clarifies the
relationship between our construction and that of Connes and
Consani \cite{CC21}. Both constructions use the same localization
process, but differ in their covering families. Connes and Consani
equip the category $\mathcal C^\infty(A)$ with a Grothendieck
topology generated by partitions involving the higher levels of
$A$, whereas we use the topology of the multiplicative monoid
spectrum
\[
X_A:=\operatorname{Spec}_{\mathrm{Deitmar}}A(1_+),
\]
whose prime ideals are not required to be additively closed.

Since the underlying category $\mathcal C^\infty(A)$
depends only on the multiplicative monoid $A(1_+)$,
the correspondence $f^\infty\mapsto D(f)$ gives an
equivalence with the category of principal open
subsets of $X_A$, with inclusions as morphisms;
see \cite[Proposition~6.2(ii)]{CC21}.
Both constructions therefore start from the same localization
presheaf $f^\infty\mapsto A_f$, but impose different sheaf
conditions.

For $A=HR$, where $R$ is a commutative ring, the additively closed
multiplicative prime ideals are precisely the ordinary ring prime
ideals. In this case, the Connes--Consani construction recovers
the classical Zariski site and the structure sheaf
$H\mathcal O_{\operatorname{Spec}R}$, as shown in
\cite[\S\S6.1,7.3]{CC21}.
For general $\mathbb F_1$-algebras, however, the additional covering
relations can change the localization presheaf upon
sheafification, preventing recovery of the original algebra
from global sections.

This phenomenon is illustrated by the Archimedean example of
\cite[\S7.6]{CC21}. Let
\[
A=|H\mathbb Q|_1,\qquad
A(n_+)=
\left\{(q_1,\ldots,q_n)\in\mathbb Q^n:
\sum_{i=1}^n|q_i|\leq1\right\}.
\]
Its multiplicative spectrum has exactly two points,
\[
\mathfrak p_0=\{0\},
\qquad
\mathfrak m=\{q\in\mathbb Q:|q|<1\},
\]
with $\mathfrak p_0$ generic and $\mathfrak m$ closed. We have
\[
D\left(\frac12\right)=\{\mathfrak p_0\},
\qquad
\mathcal O_A(X_A)=|H\mathbb Q|_1,
\qquad
\mathcal O_A(\{\mathfrak p_0\})\cong H\mathbb Q.
\]
The last isomorphism follows from
$A_{1/2}\cong H\mathbb Q$: any finite tuple of rational numbers
can be multiplied by a sufficiently large power of $\frac12$
to lie in $A(n_+)$.

The closed point $\mathfrak m$ is not additively closed:
the element $(\frac12,\frac12)\in A(2_+)$ has both components
in $\mathfrak m$, but its sum is $1\notin\mathfrak m$.
In the Connes--Consani site, this same element exhibits a
partition making the object corresponding to localization at
$\frac12$ a cover of the terminal object.
The sheaf condition forces the restriction along this covering
monomorphism to be an isomorphism. Sheafification therefore
changes the global sections from $|H\mathbb Q|_1$ to
$H\mathbb Q$.

Thus the Connes--Consani structure sheaf does not, in general,
recover the original algebra through global sections, although
the localization presheaf retains it. Our choice of covering
families allows this information to be retained in an actual
structure sheaf, with
$\Gamma(X_A,\mathcal O_A)\cong A$.

\section{The anti-equivalence}\label{sec:antiequiv}

We prove that the spectrum construction is fully
faithful and establishes an anti-equivalence between
commutative $\mathbb F_1$-algebras and absolute affine
schemes. The key point is that locality of stalk maps
allows us to recover the underlying continuous map
from the induced morphism on global sections.

\subsection{Morphisms of spectra}

We begin by showing that morphisms of $\mathbb{F}_1$-algebras induce morphisms of prime spectra:

\begin{proposition}\label{prop:morphism}
\begin{enumerate}
    \item Let $\varphi: A \to B$ be a morphism of commutative $\mathbb{F}_1$-algebras. Then $\varphi$ induces a morphism of locally $\mathbb{F}_1\mathrm{Alg}$-ed spaces:
$$(f, f^\#): \Spec B \to \Spec A.$$
    \item  Every morphism of locally $\mathbb{F}_1\mathrm{Alg}$-ed spaces:
$$(f, f^\#): \Spec B \to \Spec A$$
arises from a unique morphism of $\mathbb{F}_1$-algebras $\varphi: A \to B$.
\end{enumerate}
\end{proposition}

\begin{proof}
\begin{enumerate}
        \item Given $\varphi \in \mathrm{Hom}_{\mathbb{F}_1\mathrm{Alg}}(A, B)$, define the continuous map
    \begin{align*}
         f: |\Spec B| &\rightarrow |\Spec A|\\
         p &\mapsto \varphi^{-1}(p)
    \end{align*}
    This is continuous since $f^{-1}(D(g))=D(\varphi(g))$ for $g \in A(1_+)$.

    For each prime $p \in |\Spec B|$, localization of $\varphi$ gives
    \begin{align*}
        \varphi_p: A_{\varphi^{-1}(p)} &\rightarrow B_p\\
        n_+ \text{ level}: \displaystyle\frac{a}{s} &\mapsto \displaystyle\frac{\varphi(a)}{\varphi(s)}
    \end{align*}
    for $s \notin \varphi^{-1}(p)$.
        
For any open set $V \subseteq |\Spec A|$, we obtain an $\mathbb{F}_1$-algebra map $f^{\#}_V: \mathcal{O}_A(V) \rightarrow \mathcal{O}_B(f^{-1}(V))$ by composing $s$ with $f$ and $\varphi_p$. Concretely, for $s: q \mapsto s(q) \in A_q(n_+), \forall q \in V \subseteq |SpecA|$, we define $f^{\#}_V: \mathcal{O}_A(V) \to \mathcal{O}_B(f^{-1}(V))$ by
$$f^{\#}(s)(p) = \varphi_{p}(s(f(p)))$$
for $s \in \mathcal{O}_A(V)$ and $p \in f^{-1}(V)$. The local nature of sections ensures this is well-defined and gives a morphism of sheaves.

It remains to verify that the induced stalk maps are local
at the first level. Let $\mathfrak p\in\Spec B$, consider the stalk map
\[
f^\#_{\mathfrak p}
=\varphi_{\mathfrak p}:
A_{\varphi^{-1}(\mathfrak p)}
\longrightarrow B_{\mathfrak p},
\qquad
\frac{a}{s}\longmapsto
\frac{\varphi(a)}{\varphi(s)}.
\]
Then
\[
\frac{\varphi(a)}{\varphi(s)}
\in B_{\mathfrak p}(1_+)^\times
\quad\Longleftrightarrow\quad
\varphi(a)\notin\mathfrak p
\quad\Longleftrightarrow\quad
a\notin\varphi^{-1}(\mathfrak p)
\quad\Longleftrightarrow\quad
\frac{a}{s}
\in A_{\varphi^{-1}(\mathfrak p)}(1_+)^\times.
\]
Thus every stalk map is local at the first level, and
$(f,f^\#)$ is a morphism of locally
$\mathbb F_1\mathrm{Alg}$-ed spaces.

\item Given a morphism $(f, f^{\#}): \operatorname{Spec} B \to \operatorname{Spec} A$, taking global sections yields
$$\varphi := f^{\#}_{|\operatorname{Spec} A|}: A = \mathcal{O}_A(|\operatorname{Spec} A|) \to \mathcal{O}_B(|\operatorname{Spec} B|) = B.$$
For any $p \in |\operatorname{Spec} B|$, the stalk map $f^{\#}_{p}: \mathcal{O}_{A,f(p)} \to \mathcal{O}_{B,p}$ fits into a commutative diagram:
$$
\begin{tikzcd}
A \arrow[d, "\varphi"] \arrow[r] & A_{f(p)} \arrow[d, "f^{\#}_{p}"] \\
B \arrow[r] & B_{p}
\end{tikzcd}
$$
Since $f^{\#}_{p}$ is local at the first level, we have $f(p) = \varphi^{-1}(p)$. The commutativity of the diagram then determines $f^{\#}_{p}$ uniquely as the localization of $\varphi$, which in turn determines the entire morphism $f^{\#}$.
\end{enumerate}
\end{proof}

From the above, it is clear that the assignment $A \mapsto \Spec A$ extends to a functor:
$$\Spec: \FAlg^{\mathrm{op}} \to \AffSch_{\mathbb{F}_1}.$$

Moreover, Proposition~\ref{prop:morphism} gives a natural
bijection
$$\Hom_{\FAlg}(A, B) \cong \Hom_{\AffSch_{\mathbb{F}_1}}(\Spec B, \Spec A).$$

\begin{theorem}[Theorem~\ref{thm:main-antiequiv}]
\label{thm:antiequivalence}
The spectrum functor is an equivalence
\[
\Spec:\FAlg^{\mathrm{op}}
\xrightarrow{\sim}\AffSch_{\mathbb F_1}.
\]
Its quasi-inverse is the global sections
functor
\[
\Gamma:\AffSch_{\mathbb F_1}
\longrightarrow\FAlg^{\mathrm{op}},
\qquad
(X,\mathcal O_X)\longmapsto
\Gamma(X,\mathcal O_X).
\]
Equivalently, commutative $\mathbb F_1$-algebras
and absolute affine schemes are
anti-equivalent.
\end{theorem}

\begin{proof}
Proposition~\ref{prop:morphism} shows that
$\Spec$ is fully faithful.
It is essentially surjective by
Definition~\ref{def:absolute-affine-scheme}.
The canonical isomorphisms
\[
A\xrightarrow{\sim}\Gamma(X_A,\mathcal O_A)
\]
of Theorem~\ref{thm:stalks-and-sections}
are natural in $A$, so global sections
provide a quasi-inverse.
\end{proof}

\subsection{Examples}

\begin{example}[The absolute point]
\label{ex:absolute-point}
The spectrum of $\FF$ has a single point:
\[
|\Spec\FF|
=
\Spec_{\mathrm{Deitmar}}\{0,1\}
=
\{\mathfrak p_0\},
\qquad
\mathfrak p_0=\{0\}.
\]
Its structure sheaf satisfies
\[
\mathcal O_{\FF}(\{\mathfrak p_0\})\cong\FF.
\]
We call $\Spec\FF$ the \emph{absolute point}.

Since $\FF$ is the initial commutative
$\mathbb F_1$-algebra, $\Spec\FF$ is terminal
in $\AffSch_{\mathbb F_1}$.
The unique morphism
\[
\Spec A\longrightarrow\Spec\FF
\]
corresponds to the unit map $\FF\to A$.
\end{example}

\begin{example}[Absolute affine spaces]
\label{ex:absolute-affine-spaces}
Let
\[
M_n
=
\{0\}\sqcup
\left\{
T_1^{a_1}\cdots T_n^{a_n}
\;\middle|\;
a_1,\ldots,a_n\in\NN
\right\}
\]
be the free pointed commutative monoid on
$n$ generators, with multiplication of
monomials and absorbing element $0$.
Set
\[
\FF[T_1,\ldots,T_n]:=\mathbb S M_n,
\qquad
\mathbb A^n_{\mathbb F_1}
:=
\Spec\FF[T_1,\ldots,T_n].
\]

The prime ideals of $M_n$ are indexed by
subsets $J\subseteq\{1,\ldots,n\}$:
\[
\mathfrak p_J
=
\{0\}\cup\bigcup_{j\in J}T_jM_n.
\]
Indeed, a prime ideal contains a monomial
if and only if it contains at least one
generator occurring in that monomial.
Conversely, the complement of each
$\mathfrak p_J$ is multiplicatively closed.
Thus $\mathbb A^n_{\mathbb F_1}$ has $2^n$
points.

In particular, the absolute affine line
has two points
\[
\eta=\{0\},
\qquad
\mathfrak m=(T),
\]
with $\eta$ generic and $\mathfrak m$ closed.
Its nonempty open sets are $\{\eta\}=D(T)$
and the whole space, with sections
\[
\mathcal O(\{\eta\})
\cong\FF[T,T^{-1}],
\qquad
\mathcal O(\{\eta,\mathfrak m\})
\cong\FF[T].
\]
Where
\[
\FF[T,T^{-1}]
:=
\mathbb S\bigl(\{0\}\sqcup\{T^k:k\in\ZZ\}\bigr).
\]
\end{example}

\begin{example}[Eilenberg--MacLane algebras]
\label{ex:em-spectrum}
Let $R$ be a commutative ring.
The underlying space of $\Spec(HR)$ is
the monoid spectrum of $R$:
\[
X_{HR}=\Spec_{\mathrm{Deitmar}}(R,\cdot).
\]
Its points are the proper multiplicative
ideals whose complements are multiplicatively
closed.

For example, for $R=\ZZ$, let $\mathcal P$ denote the
set of positive prime numbers.
Every monoid prime ideal has the form
\[
\mathfrak p_J
=
\{0\}\cup\bigcup_{p\in J}p\ZZ,
\qquad
J\subseteq\mathcal P.
\]
Indeed, primality and factorization of
nonzero integers show that a prime ideal
is determined by the prime numbers it
contains. Conversely, every displayed
subset is a proper multiplicative ideal.

The choice $J=\varnothing$ gives $\{0\}$,
and singleton subsets give the nonzero
ordinary ring prime ideals.
The largest monoid prime is
\[
\mathfrak p_{\mathcal P}
=
\ZZ\setminus\{-1,1\}.
\]

For every integer $f$, the structure sheaf
satisfies
\[
\mathcal O_{H\ZZ}(D(f))
\cong(H\ZZ)_f
\cong H(\ZZ_f).
\]
Thus sections over principal opens are
still described by classical localization,
although the underlying space and its
covering families differ from those of
$\Spec\ZZ$.

For example, the ordinary Zariski opens
$D(2)$ and $D(3)$ cover $\Spec\ZZ$.
The corresponding opens in $X_{H\ZZ}$
do not cover: the monoid prime
\[
2\ZZ\cup3\ZZ
\]
belongs to neither.
The base change construction in the next
section recovers the classical affine
scheme $\Spec R$ from $\Spec(HR)$.
\end{example}

\section{Base change to classical schemes}
\label{sec:applications}

We describe the passage from absolute affine schemes
to classical affine schemes. The algebraic adjunction
between extension of scalars and the Eilenberg--MacLane
functor induces a geometric adjunction in the opposite
direction. In particular, every classical affine scheme
admits a functorial absolute affine model whose base
change recovers the original scheme.

\subsection{Extension of scalars}

Recall from \cite{xu-hyper} that the Eilenberg--MacLane
functor
\[
H:\CRing\longrightarrow\FAlg
\]
admits a left adjoint $-\otimes_{\FF}\ZZ$.
Explicitly, $-\otimes_{\FF} \ZZ$ is the functor given by:
$$A \otimes_{\FF} \ZZ := \ZZ[A(1_+)] / \mathcal{R}$$
where $\ZZ[A(1_+)]$ is the monoid ring of the
underlying multiplicative monoid, and $\mathcal{R}$ is the ideal generated by the following relations:
\begin{enumerate}[label=(\roman*)]
\item \emph{(Basepoint relation)} $[*] = 0$.
\item \emph{(Additivity relations)} $[a] + [b] = [c]$ for all $a, b, c \in A(1_+)$ with $c \in a \oplus b$. (See Remark \ref{rem:prime} for the operation $\oplus$)
\end{enumerate}
and morphisms canonically determined by universal property. One can interpret the Eilenberg-Maclane functor $H$ as a forgetful/restriction of scalars functor, and $-\otimes_{\FF} \ZZ$ as an extension of scalars functor from $\FF$ to $\ZZ$.
This construction is characterized by natural
bijections
\[
\Hom_{\CRing}(A\otimes_{\FF}\ZZ,R)
\cong
\Hom_{\FAlg}(A,HR).
\]
A morphism $\varphi:A\to B$ induces a ring
homomorphism
\[
\varphi\otimes_{\FF}\ZZ:
A\otimes_{\FF}\ZZ
\longrightarrow
B\otimes_{\FF}\ZZ,
\qquad
[a]\longmapsto[\varphi(a)].
\]
Thus extension of scalars defines a functor
\[
-\otimes_{\FF}\ZZ:\FAlg\longrightarrow\CRing.
\]

\subsection{The geometric adjunction}

Through the affine anti-equivalences, extension
of scalars induces a base change functor, which we still denote by the same notation
\[
-\otimes_{\FF}\ZZ:
\AffSch_{\mathbb F_1}\longrightarrow\AffSch_{\ZZ},
\qquad
\Spec A\longmapsto\Spec(A\otimes_{\FF}\ZZ).
\]
Likewise, the Eilenberg--MacLane functor induces
\[
H:\AffSch_{\ZZ}\longrightarrow\AffSch_{\mathbb F_1},
\qquad
\Spec R\longmapsto\Spec(HR).
\]

\begin{theorem}\label{thm:base-change-schemes}
Base change is right adjoint to the functor
induced by $H$. Explicitly, there is a natural
bijection
\[
\Hom_{\AffSch_{\ZZ}}
\bigl(\Spec R,\Spec(A\otimes_{\FF}\ZZ)\bigr)
\cong
\Hom_{\AffSch_{\mathbb F_1}}
\bigl(\Spec(HR),\Spec A\bigr).
\]
\end{theorem}

\begin{proof}
The classical affine anti-equivalence,
the algebraic adjunction, and
Theorem~\ref{thm:antiequivalence} give
\[
\begin{aligned}
\Hom_{\AffSch_{\ZZ}}
\bigl(\Spec R,\Spec(A\otimes_{\FF}\ZZ)\bigr)
&\cong
\Hom_{\CRing}(A\otimes_{\FF}\ZZ,R)\\
&\cong
\Hom_{\FAlg}(A,HR)\\
&\cong
\Hom_{\AffSch_{\mathbb F_1}}
\bigl(\Spec(HR),\Spec A\bigr).
\end{aligned}
\]
These bijections are natural in $A$ and $R$.
\end{proof}

\begin{corollary}\label{cor:classical-affine-models}
The induced functor
\[
H:\AffSch_{\ZZ}\longrightarrow\AffSch_{\mathbb F_1}
\]
is fully faithful. For every commutative ring
$R$, there is a natural isomorphism
\[
(\Spec(HR))\otimes_{\FF}\ZZ
\cong\Spec R.
\]
\end{corollary}

\begin{proof}
Full faithfulness follows from that of
$H:\CRing\to\FAlg$ and the affine
anti-equivalences.
Since $H$ is a fully faithful right adjoint
at the algebraic level, the counit is an
isomorphism
\[
(HR)\otimes_{\FF}\ZZ\xrightarrow{\sim}R,
\qquad
[r]\longmapsto r.
\]
Passing to spectra gives the result.
\end{proof}

The embedding induced by $H$ generally changes
the underlying topological space, since
$\Spec(HR)$ uses the multiplicative monoid
spectrum of $R$. The classical affine scheme
is recovered after base change.

\subsection{Examples and comparison with monoid schemes}

For a pointed commutative monoid $M$, the
spherical algebra $\mathbb S M$ encodes only
the additive relations involving zero.
Consequently,
\[
(\mathbb S M)\otimes_{\FF}\ZZ
\cong
\ZZ_0[M]:=\ZZ[M]/([0]).
\]
Thus base change on spherical algebras agrees
with the usual base change of affine monoid
schemes.

\begin{example}[Base change of the absolute point]
\label{ex:base-change-point}
For the pointed monoid $\{0,1\}$,
the contracted monoid ring is $\ZZ$.
Hence
\[
(\Spec\FF)\otimes_{\FF}\ZZ
\cong\Spec\ZZ.
\]
The absolute point becomes the terminal
classical affine scheme.
\end{example}

\begin{example}[Base change of absolute affine spaces]
\label{ex:affine-line-base-change}
\label{ex:affine-n-space}
Let $M_n$ be the free pointed commutative
monoid on $T_1,\ldots,T_n$, as in
Example~\ref{ex:absolute-affine-spaces}.
Then
\[
\ZZ_0[M_n]\cong\ZZ[T_1,\ldots,T_n],
\]
so
\[
\mathbb A^n_{\mathbb F_1}\otimes_{\FF}\ZZ
\cong
\Spec\ZZ[T_1,\ldots,T_n]
=
\mathbb A^n_{\ZZ}.
\]
In particular, the absolute affine line
becomes $\mathbb A^1_{\ZZ}$.
\end{example}

\begin{example}[Eilenberg--MacLane algebras]
\label{ex:fixed-points}
For a commutative ring $R$, the relations
defining $(HR)\otimes_{\FF}\ZZ$ include
\[
[a]+[b]=[a+b]
\qquad(a,b\in R).
\]
Together with the multiplicative relations,
these recover the ring structure of $R$,
giving
\[
(HR)\otimes_{\FF}\ZZ\cong R.
\]
For instance,
\[
(\Spec(H\mathbb F_p))\otimes_{\FF}\ZZ
\cong\Spec\mathbb F_p.
\]
\end{example}

The last example also distinguishes the range
of this construction from that of affine
monoid schemes. The base change of an affine
monoid scheme is the spectrum of a contracted
monoid ring $\ZZ_0[M]$, whose underlying
additive group is free abelian on
$M\setminus\{0\}$. Such a ring cannot be
isomorphic to $\mathbb F_p$.

In contrast, the higher-level structure of $HR$ records
the additive relations needed to recover an arbitrary
commutative ring $R$.  Thus every classical
affine scheme occurs as the base change
of a functorially associated absolute affine
scheme, extending the range beyond those
arising from affine monoid schemes.


\bibliographystyle{amsalpha}  
\bibliography{ref}

\end{document}